\documentclass[onethmnum,onefignum,final]{siamltex}

\usepackage{amsmath}
\usepackage{amssymb}
\usepackage{amsfonts}
\usepackage{mathrsfs}
\usepackage{bbm}
\usepackage{graphicx,color}
\usepackage{enumerate}
\usepackage{nicefrac}
\usepackage{url}
\usepackage{chngcntr}
\counterwithin{figure}{section}
\counterwithin{theorem}{section}

\newcommand{\N}{\mathbb{N}} 
\newcommand{\R}{\mathbb{R}}

\newcommand{\C}{\mathbb{C}}
\newcommand{\curly}[1]{\left\lbrace #1 \right\rbrace}
\newcommand{\mybrace}[1]{\left( #1 \right)}
\newcommand{\sd}[1]{\ \left| \ #1 \right.}
\newcommand{\betrag}[1]{\left| #1 \right|}
\newcommand{\norm}[1]{\left\| #1\right\| }

\newcommand{\vonbis}[3]{#1=#2,\ldots,#3}
\newcommand{\kgkg}[3]{#1\leq#2\leq#3}

\newcommand{\Realteil}[1]{\mathrm{Re}(#1)}
\newcommand{\Imagteil}[1]{\mathrm{Im}(#1)}
\newcommand{\mile}[3]{\mathrm{E}_{#1,#2}(#3)}

\newcommand{\Lapllam}[1]{\mathscr{L}\curly{#1,\ \lambda}}
\newcommand{\invLapl}[1]{\mathscr{L}^{-1}\curly{#1,\ t}}
\newcommand{\CDiff}[2]{{\mathrm{D}_{t}^{#2}}#1}

\newcommand{\bloperator}[1]{\mathcal{B}(#1)}
\newcommand{\res}[2]{R\mybrace{#1,\,#2}}

\newcommand{\sector}[1]{\Sigma_{#1}}
\newcommand{\lgro}[1]{\left[ #1\right) }
\newcommand{\loro}[1]{\left(#1\right)}
\newcommand{\lorg}[1]{\left(#1\right]}
\newcommand{\lgrg}[1]{\left[#1\right]}

\newcommand{\Cfkt}[2]{C^{#1}\!\mybrace{#2}}
\newcommand{\Rplus}{\R_+}

\newcommand{\Gammafkt}[1]{\Gamma\mybrace{#1}}
\newcommand{\convolution}[3]{\mybrace{#1*#2}(#3)}
\newcommand{\e}{\mathrm{e}}
\newcommand{\dint}{\,\mathrm{d}}

\newcommand{\Id}{\mathrm{Id}}
\newcommand{\Sa}{S_\alpha}
\newcommand{\Pa}{P_\alpha}
\newcommand{\Q}{\mathcal{A}}
\newcommand{\eins}{\mathbbm{1}}
\newcommand{\Oh}[1]{\mathcal{O}\mybrace{#1}}
\newcommand{\Il}{I_\ell}
\newcommand{\eps}{\texttt{eps}}
\newcommand{\arccosh}{\mathrm{arccosh}}
\newcommand{\un}[1]{u_N^{(#1)}}
\newcommand{\Un}[1]{U_N^{(#1)}}
\newcommand{\lamkl}{\lambda_k^{(\ell)}}
\newcommand{\sa}{\lambda^\alpha}
\newcommand{\I}{i}
\newcommand{\Gl}{\Gamma_\ell}

\newtheorem{remark}[theorem]{Remark}
\newtheorem{assumption}[theorem]{Assumptions}
\newtheorem{defi}[theorem]{Definition}

\makeatletter\@addtoreset{equation}{section}\makeatother
\graphicspath{ {images/}{images}}

\usepackage{hyperref}

\begin{document}

\title{Fast and Parallel Runge-Kutta Approximation of Fractional Evolution Equations}

\author{Marina Fischer\thanks{Mathematisches Institut, Heinrich-Heine-Universit\"at, 40225 D\"usseldorf, Germany ({\tt marina.fischer@uni-duesseldorf.de}).}}

\maketitle

\begin{abstract}
We consider a linear inhomogeneous fractional evolution equation which is obtained from a Cauchy problem by replacing its first-order time derivative with Caputo's fractional derivative. The operator in the fractional evolution equation is assumed to be sectorial. 
By using the inverse Laplace transform a solution to the fractional evolution equation is obtained which can be written as a convolution. 
Based on $L$-stable Runge-Kutta methods a convolution quadrature is derived which allows a stable approximation of the solution. Here, the convolution quadrature weights are represented as contour integrals. On discretising these integrals, we are able to give an algorithm which computes the solution after $N$ time steps with step size $h$ up to an arbitrary accuracy $\varepsilon$. For this purpose the algorithm only requires $\mathcal{O}(N)$ Runge-Kutta steps for a large number of scalar linear inhomogeneous ordinary differential equations and the solutions of $\mathcal{O}(\log(N)\log(\nicefrac{1}{\varepsilon}))$ linear systems what can be done in parallel. In numerical examples we illustrate the algorithm's performance.
\end{abstract}



\section{Introduction}
Over the last few decades the interest in linear fractional differential equations has grown, not least due to the fact that they model phenomena in applied mathematics and physics such as anomalous diffusion in viscoelastic materials. See the references in \cite{Cuesta2006, Garrappa2015, Jiang2017} for an overview of applications and examples. In this paper, we study the inhomogeneous fractional evolution equation 
\begin{equation}\label{eqDiffglg}
\CDiff{u(t)}{\alpha}=A u(t)+g(t)\quad\text{for}\quad t\in I:=\loro{0,T}\quad\text{and}\quad u(0)=u_0\in X,
\end{equation}
where $A:D(A)\subset X\rightarrow X$ is a closed linear densely defined operator in a Banach space $\mybrace{X,\norm{.}}$ and the Caputo fractional derivative $\CDiff{}{\alpha}$ of order $\alpha\in\loro{0,1}$ is given by
\begin{align*}
\CDiff{u(t)}{\alpha}:=\frac{1}{\Gammafkt{1-\alpha}}\int_{0}^{t}(t-\tau)^{-\alpha}\,u'(\tau)\dint\tau;\quad\text{cf.}\ \cite{Mainardi2012, Podlubny1999}.
\end{align*}
Although it is possible to attribute physical meaning to the initial conditions for other definitions of fractional derivative, such as the one of Riemann-Liouville, using Caputo's derivative enables us to take into account initial values whose physical interpretation is easier to see; cf. \cite{Podlubny2006, Podlubny1999}. 

We further assume the operator $A$ in \eqref{eqDiffglg} to be \textit{sectorial}; cf. \cite{Engel2000, Henry1981} which means that: 
\begin{equation}\label{defSectorial}
\begin{minipage}{0.9\textwidth}
\centering
The resolvent $\res{\lambda}{A}$ is analytic in a sector \\
$\sector{\theta_0}:=\curly{\lambda\in\C\setminus\curly{0}\sd{\betrag{\arg(\lambda)}<\nicefrac{\pi}{2}+\theta_0}}$ with  $\theta_0\in\lorg{0,\nicefrac{\pi}{2}}$\\ and is bounded there by $\norm{\res{\lambda}{A}}\leq \nicefrac{C}{\betrag{\lambda}}$ for some real $C>0$.
\end{minipage}
\end{equation}
The assumption includes unbounded operators such as the Laplacian $A=\Delta$ on $\R^d$ or on a domain $\Omega\subset \R^d$ together with Dirichlet or Neumann boundary conditions \cite{Henry1981,Pazy1983}. Typically, the spatial discretisation of sectorial operators with finite differences or finite elements also satisfies the bound \eqref{defSectorial} in $\mathrm{L}^p$-norms, where the constant $C$ and the angle $\theta_0$ are independent of the spatial discretisation parameter; see e.g. \cite{Ashyralyev1994, Bakaev2002}.

On the one hand, fractional evolution equations as \eqref{eqDiffglg} with sectorial operators have already been studied under theoretical aspects, for example in order to give existence and uniqueness theorems concerning their solutions; cf. \cite{Bajlekov2001, Chen2010, Li2012}. Especially \cite{Bajlekov2001} should be mentioned as one of the most prominent sources of this paper's theoretical basis as it focuses on the homogeneous version of \eqref{eqDiffglg} as well as on the evolution equation using the Riemann-Liouville derivative.\\ 
On the other hand, there is a huge amount of numerical methods for the time discretisation of \eqref{eqDiffglg} proposed by various authors \cite{Baffet2017,Banjai2011,Cuesta2006,Diethelm2006,Garrappa2015,Jacobs2016,Jiang2017,Lopez-Fernandez2008,Lubich2002,McLean2010,Schaedle2006,Stynes2017,Weideman2007}. In almost every approach integrals occur which have to be approximated throughout the numerical treatment of \eqref{eqDiffglg}. Some of these are contour integrals which for example can be derived using the Laplace transform of the fractional equation; see e.g. \cite{Weideman2007}. Other approaches make use of the integral representation of fractional derivative or fractional integration \cite{Garrappa2015,Jiang2017}. Often the different approaches lead to convolution integrals. Here, among methods like product integration rules \cite{Garrappa2015} or the recently formulated sum-of-exponentials approximation \cite{Jiang2017}, especially convolution quadrature based on multistep \cite{Lubich1986} or Runge-Kutta methods \cite{Lubich1993} has to be mentioned as it provides a theoretical background for the numerical treatment of fractional equations such as \eqref{eqDiffglg}.

In this paper, we give the solution $u(t)$ to \eqref{eqDiffglg} which can partly be represented as a convolution integral. Based on the fast Runge-Kutta approximation of inhomogeneous parabolic equations, which is presented in \cite{Lopez-Fernandez2005}, we propose an algorithm that computes $u(t)$ at a fixed time $t=Nh$ after $N$ steps of Runge-Kutta convolution quadrature and with step size $h$, up to an arbitrary accuracy $\varepsilon$. In order to do so the algorithm requires 
\begin{equation*}
\Oh{N}\ \text{steps}
\end{equation*}
of an implicit $L$-stable Runge-Kutta time discretisation for ordinary differential equations of the form $y'(t)= \lambda y(t)+g(t)$ what can be done in parallel for $\Oh{\log(\nicefrac{1}{\varepsilon})}$ complex parameters $\lambda$ and for the entries of the inhomogeneity $g$. Furthermore, the algorithm requires the solution of only 
\begin{equation*}
\Oh{\log(N)\log(\nicefrac{1}{\varepsilon})}\ \text{linear systems}
\end{equation*}
of the form $(\lambda^\alpha\Id-A)x=y$, all of which can be treated in parallel. If the number of steps $N$ is large, the number of linear systems is thus reduced noticeably.

Therefore, the algorithm especially offers a fast approximation to the solution of \eqref{eqDiffglg} on a short subinterval around the fixed time $t$ or at a relatively small number of selected time points. However, it is not useful for computing all values $u_1,\ldots,u_N$; see e.g. \cite{Garrappa2015,Jacobs2016,Jiang2017,Lubich1988,Lubich1993,Lubich2002,Schaedle2006} for algorithms with this purpose.

Before focussing on numerical considerations related to \eqref{eqDiffglg}, Section \ref{chSol} is concerned with finding the solution of \eqref{eqDiffglg} itself. Using some theoretical results on solution operators, we find that the solution contains a convolution of an operator-valued function with the inhomogeneity $g$. Hence, in Section \ref{chConv} we review convolution quadrature based on Runge-Kutta methods. Using Cauchy's integral formula we give a contour integral representation of the occurring convolution weights whose discretisation along hyperbolas is studied in Section \ref{chDiscret}. Based on the results of these sections, we describe the fast and parallel algorithm in Section \ref{chAlgo} and give an extension to systems with a mass matrix. Section \ref{chNumEx} concludes the paper by illustrating the algorithm's performance in some numerical experiments.

\section{Solution to the fractional evolution equation}\label{chSol}
In this first basic section we take a closer look on the fractional evolution equation and deal with its strong solution. To this end, by making use of the resolvent, we define two operators $\Sa$ and $\Pa$ that enable us to write the solution with the help of a convolution which is essential for the algorithm we propose here. 
\subsection{Preliminaries}
We first have a look on the operator-valued function \linebreak$\res{\lambda^\alpha}{A}$ that plays an important role in giving the solution to equation \eqref{eqDiffglg}: Define the angle $\theta_1$ in dependence on the order $\alpha$ of the fractional derivative and on the angle $\theta_0$ as
\begin{equation}\label{eqWahltheta1}
\theta_1:=\min\curly{\frac{\pi(1-\alpha)+2\theta_0}{2\alpha},\,\frac{\pi}{2}}\ \in\lorg{0,\frac{\pi}{2}}.
\end{equation}
With this choice, on taking the $\alpha$th power of $\lambda\in\sector{\theta_1}$ we obtain that $\lambda^\alpha$ is contained in the resolvent's sector of analyticity $\sector{\theta_0}$. Therefore, $\res{\lambda^\alpha}{A}$ is analytic in the sector $\sector{\theta_1}$. Here the relation $\theta_1\geq\theta_0$ yields $\sector{\theta_0}\subset\sector{\theta_1}$. Taking into account the bound \eqref{defSectorial}, $\res{\lambda^\alpha}{A}$ thus satisfies
\begin{align}\label{eqAbschSaPa}
\norm{\res{\lambda^\alpha}{A}}\leq \frac{C}{\betrag{\lambda}^\alpha}\quad\text{and}\quad
\norm{\lambda^{\alpha-1}\res{\lambda^\alpha}{A}}\leq \frac{C\betrag{\lambda}^{\alpha-1}}{\betrag{\lambda}^\alpha}=\frac{C}{\betrag{\lambda}}
\end{align}
for some real $C>0$ and for $\lambda\in\sector{\theta_1}$.

\subsection{Homogeneous fractional evolution equation}
In the next two subsections we aim at giving the solution to the  evolution equation \eqref{eqDiffglg} whose approximation is studied in the main part of this paper. Here, for shortness, we introduce the function $\Phi_\beta(t)$ which we define for a parameter $\beta\geq 0$ as
\begin{equation*}
\Phi_\beta(t):=
\begin{cases}
\quad \frac{1}{\Gammafkt{\beta}}\,t^{\beta-1},\quad &t>0,\\
\quad0, &t\leq0,
\end{cases}
\end{equation*}
where $\Gammafkt{\beta}$ is the Gamma function \cite{Mainardi2012}. It can be verified that this function satisfies the semigroup property
$\convolution{\Phi_\alpha}{\Phi_\beta}{t}=\Phi_{\alpha+\beta}(t)$ for $\alpha,\ \beta\geq 0.$ 

\begin{defi}\label{defStrongSol}
	A function $u\in\Cfkt{}{I,X}$ that satisfies $u\in\Cfkt{}{I,D(A)}$ and ${\Phi_{1-\alpha}*\mybrace{u-u_0}\in\Cfkt{1}{I,X}}$ and that fulfils \eqref{eqDiffglg} is called a strong solution to equation \eqref{eqDiffglg}; cf. \cite{Bajlekov2001}.
\end{defi}

In order to give a representation of the solution we first consider the homogeneous version of the fractional evolution equation \eqref{eqDiffglg} which reads as
\begin{equation}\label{eqHomoDiffglg}
\CDiff{u(t)}{\alpha}=A u(t)\quad\text{for}\quad t\in I\quad\text{with}\quad u(0)=u_0.
\end{equation}
Applying the Riemann-Liouville fractional integral \cite{Mainardi2012}, that means convolving each side of the evolution equation with $\Phi_\alpha$, we find that \eqref{eqHomoDiffglg} is equivalent to the Volterra equation
\begin{equation*}
u(t)=u_0+\convolution{\Phi_\alpha}{Au}{t}\quad\text{for}\quad t\in I.
\end{equation*}
Now we can give the following important definition; cf. \cite{Bajlekov2001, Pruess1993}.
\begin{defi}\label{defSolOp}
	A family $\curly{\Sa(t)}_{t\geq0}\subset\bloperator{X}$ is called a solution operator for \eqref{eqHomoDiffglg} if the following conditions are satisfied:
	\begin{itemize}
		\item[(a)] $\Sa(t)$ is strongly continuous for $t\geq0$ and $\Sa(0)=\Id$.
		\item[(b)] $\Sa(t)D(A)\subset D(A)$ and $A\Sa(t)u_0=\Sa(t)Au_0$ for all $u_0\in D(A)$ and $t\geq 0$.
		\item[(c)] The resolvent equation 
		\begin{equation}\label{eqVolterra}
		\Sa(t)u_0=u_0+\convolution{\Phi_{\alpha}}{A\Sa u_0}{t}
		\end{equation}
	holds for all $u_0\in D(A)$ and $t\geq 0$.
	\end{itemize}
	If $\Sa(t)$ moreover admits an analytic extension to a sector $\sector{\rho_0-\nicefrac{\pi}{2}}$ for some ${\rho_0\in\lorg{0,\nicefrac{\pi}{2}}}$ and if for each $\rho<\rho_0$ and $\omega>0$ there is a constant $C>0$ such that the bound
	\begin{equation*}
	\norm{\Sa(t)}\leq C\,\e^{\omega\Realteil{t}}\quad \text{for}\quad t\in\sector{\rho-\nicefrac{\pi}{2}}
	\end{equation*}
	is fulfilled, then it is of analyticity type $(\rho_0,0)$.
\end{defi}

In \cite[Prop. 1.1]{Pruess1993} it is shown that \eqref{eqHomoDiffglg} is well-posed in the sense of \cite[Def. 2.2]{Bajlekov2001} iff it admits a solution operator. Furthermore, if a solution operator exists, then it is unique. In \cite[Thm. 2.14]{Bajlekov2001} the existence of solution operators for \eqref{eqHomoDiffglg} is characterized as follows: 
\begin{theorem}
	A linear closed densely defined operator $A$ generates an analytic solution operator $\Sa(t)$ of type $(\rho_0,0)$ with $\rho_0\in\lorg{0,\nicefrac{\pi}{2}}$ iff $\lambda^\alpha\in\rho(A)$ for each $\lambda\in\sector{\rho_0}$ and for any $\omega>0,\ \rho<\rho_0$, there is a constant $C>0$ such that
	\begin{equation*}
	\norm{\lambda^{\alpha-1}\res{\lambda^\alpha}{A}}\leq \frac{C}{\betrag{\lambda}}\quad\text{for}\quad \lambda\in\sector{\rho}.
	\end{equation*}
\end{theorem}

With the findings in the preliminaries, we can thus deduce that \eqref{eqHomoDiffglg} has a unique solution operator $\curly{\Sa(t)}_{t\geq0}$ which is of analyticity type $(\theta_1,0)$ and which, as proven in \cite{Bajlekov2001}, can be written as the inverse Laplace transform \cite{Doetsch1970, Spiegel1977} of the operator-valued function $\res{\lambda^\alpha}{A}$. That means, it is
\begin{align*}
\Sa(t)&:=\invLapl{\lambda^{\alpha-1}\res{\lambda^\alpha}{A}}\\
& = \frac{1}{2\pi\I}\int_\Gamma \e^{\lambda t}\,\lambda^{\alpha-1}\res{\lambda^\alpha}{A}\dint\lambda,
\end{align*}
with a contour $\Gamma$ in $\sector{\theta_1}$, going to infinity with an acute angle to the negative real half-axis and oriented counter-clockwise.
It can be verified that we have $\Sa(t)u_0\in D(A)$ for any $u_0\in X$ and $t>0$. Besides, the bound
\begin{equation*}
\norm{A\Sa(t)}\leq C\,\e^{\omega t}\,(1+t^{-\alpha})
\end{equation*}
holds for $t>0$, $C>0$ and $\omega>0$ \cite[Prop. 2.15]{Bajlekov2001}. Now it is easy to see that $\Sa(t)u_0$  satisfies the conditions in Definition \ref{defStrongSol} and hence is a strong solution to \eqref{eqHomoDiffglg} for $t\in I$ and for each $u_0\in X$.

\subsection{Inhomogeneous fractional evolution equation}
After having discussed the homogeneous case, we now focus on the inhomogeneous evolution equation
\begin{equation}\label{eqInhomoDiffglg}
\CDiff{u(t)}{\alpha}=A u(t) + g(t)\quad\text{for}\quad t\in I\quad\text{with}\quad u(0)=0.
\end{equation}
In order to give its strong solution we define another operator $\Pa(t)\in\bloperator{X}$ for $t\geq 0$ as
\begin{align}
\Pa(t)&:=\invLapl{\res{\lambda^\alpha}{A}}\notag\\
& = \frac{1}{2\pi\I}\int_\Gamma \e^{\lambda t}\,\res{\lambda^\alpha}{A}\dint\lambda,\label{eqPa}
\end{align} 
where again $\Gamma$ is a counter-clockwise oriented contour contained in $\sector{\theta_1}$, going to infinity with an acute angle to the negative real half-axis. This definition is possible due to the bound \eqref{eqAbschSaPa} which also yields $\norm{\Pa(t)}\leq C t^{\alpha-1}\e^{\omega t}$ for $t>0$ and $\omega> 0$; cf. \cite{Doetsch1970}.
 
Although $\Pa(t)$ is not a solution operator in the sense of Definition \ref{defSolOp} it has some properties in common with $\Sa(t)$ which we list next.
\begin{lemma}\label{lemPa}
	The following assertions hold for the operator $\Pa(t)$.
	\begin{itemize}
		\item[(a)] $\Pa(t)$ is strongly continuous for $t \geq 0$.
		\item[(b)] $\Pa(t)u_0\in D(A)$ for $t>0$ and for all $u_0\in X$.
		\item[(c)] $A\Pa(t)u_0=\Pa(t)Au_0$ for all $u_0\in D(A)$.
	\end{itemize}
\end{lemma}
Here, properties (a) and (b) are shown in \cite{Bajlekov2001}. Helpful for a proof of property (c) is Proposition 1.1.7 in \cite{Arendt2001}.

\begin{theorem}\label{thmInhomLsg}
	Assume that $g\in\Cfkt{1}{I,D(A)}$.
	Then the unique strong solution to \eqref{eqInhomoDiffglg} is given by
	\begin{equation}\label{eqLosu}
	u(t)=\convolution{\Pa}{g}{t}\quad \text{for}\quad t\in I.
	\end{equation}
\end{theorem}
\begin{proof}(cf. \cite{Li2012,Pazy1983})
	Uniqueness: Let $u_1, \ u_2\in\Cfkt{}{I,X}$ be two different strong solutions to \eqref{eqInhomoDiffglg}. Then $w:=u_1-u_2\in\Cfkt{}{I,X},\ w(0)=0$ and $\CDiff{w(t)}{\alpha}=Aw(t)$ for $t\in I$. The unique strong solution is thus given by $w(t)=\Sa(t)w(0)=0$ which yields $u_1=u_2.$\\
	Existence: We show that $u(t)$ satisfies the properties in Definition \ref{defStrongSol}. It follows from Lemma \ref{lemPa} and \cite[Prop. 1.3.4]{Arendt2001} that $\Pa*g\in\Cfkt{}{I,D(A)}$.
	Using the convolution theorem for Laplace transforms \cite{Doetsch1970} the equality $\Phi_{1-\alpha}*\Pa*g=\Sa*g$ can easily be verified. On using \cite[Prop. 1.2]{Pruess1993} and since $g\in\Cfkt{1}{I,D(A)}$ we have that $\Sa*g$, and therefore $\Phi_{1-\alpha}*\Pa*g$, is in $\Cfkt{1}{I,D(A)}$. The strong continuity of $\Pa(t)$ implies $u(0)=0$.\\
	Laplace transforming the solution $u(t)$ and using the convolution theorem yields for $\Realteil{\lambda}>0$
		\begin{equation*}
		\Lapllam{u(t)}=\Lapllam{\Pa(t)}\Lapllam{g(t)}=\res{\lambda^\alpha}{A}G(\lambda),
		\end{equation*}
		where $G(\lambda)$ denotes the Laplace transform of $g(t)$.
		Because of the equivalence between \eqref{eqInhomoDiffglg} and the Volterra equation $u(t)=\convolution{\Phi_\alpha}{g}{t}+\convolution{\Phi_\alpha}{Au}{t}$, we know from \cite[Prop. 1.2]{Pruess1993} that \eqref{eqInhomoDiffglg} is satisfied by
		\begin{equation*}
		\widetilde{u}(t)=\convolution{\Phi_\alpha}{\Sa g(0)}{t}+(\Phi_\alpha*\Sa*g')(t)\quad\text{for}\quad t\in I,
		\end{equation*}
		such that for $\Realteil{\lambda}>0$ we have 
		\begin{align*}
		&\Lapllam{\convolution{\Phi_\alpha}{\Sa g(0)}{t}+(\Phi_\alpha*\Sa*g')(t)}\\
		&=\Lapllam{\Phi_\alpha(t)}\Lapllam{\Sa(t)}g(0)+\Lapllam{\Phi_\alpha(t)}\Lapllam{\Sa(t)}\Lapllam{g'(t)}\\
		&=\lambda^{-1}\res{\lambda^\alpha}{A}g(0)+\lambda^{-1}\res{\lambda^\alpha}{A}(-g(0)+\lambda G(\lambda))\\
		&=\res{\lambda^\alpha}{A}G(\lambda).
		\end{align*}
		It follows from the uniqueness of the Laplace transform that $u(t)=\widetilde{u}(t)$ and hence $u(t)$ satisfies \eqref{eqInhomoDiffglg}.
\end{proof}

Combining the results concerning the strong solutions of \eqref{eqHomoDiffglg} and \eqref{eqInhomoDiffglg} we can thus deduce that the strong solution to \eqref{eqDiffglg} is given by
\begin{equation*}
u(t)=\Sa(t)u_0+\convolution{\Pa}{g}{t}\quad\text{for}\quad t\in I.
\end{equation*}\pagebreak
\begin{remark}
	\begin{itemize}
		\item[(a)] In the case $A\in\bloperator{X}$ this corresponds to the well-known formula
		\begin{align*}
		u(t)=\mile{\alpha}{1}{At^\alpha}u_0+\int_{0}^{t}(t-\tau)^{\alpha-1}\mile{\alpha}{\alpha}{A(t-\tau)^\alpha}g(\tau)\dint\tau;\quad\text{cf. \cite{Podlubny1999}},
		\end{align*}
		where $\mile{\alpha}{\beta}{t}$ with $\alpha,\beta>0$ is the Mittag-Leffler function; see e.g. \cite{Mainardi2012, Podlubny1999}.
		\item[(b)] Instead of assuming $g\in\Cfkt{1}{\Rplus,D(A)}$ in Theorem \ref{thmInhomLsg} it is even sufficient to consider $g\in W^{1,1}(I,D(A))$; cf. \cite[Prop. 1.2]{Pruess1993}.
	\end{itemize}
\end{remark}

\section{Approximation of the convolution}\label{chConv}
In this section we provide the basis for the computations in the fast algorithm. In order to do so, we review briefly Runge-Kutta schemes and Runge-Kutta based convolution quadrature and give two representations of the convolution quadrature weights. 
\subsection{Runge-Kutta schemes}
In the following we consider an implicit Runge-Kutta method with $s$ stages and coefficients $a_{ij},\ b_j,\ c_i$ for $\vonbis{i,j}{1}{s}$, which is of (classical) order $p\geq 1$ and stage order $\widetilde{p}\leq p$. We denote the Runge-Kutta matrix by $\Q=\mybrace{a_{ij}}_{i,j=1}^s$ and the row vector of the weights by $b^T=\mybrace{b_j}_{j=1}^s$. Hence, the stability function is given by
\begin{align*}
r(z)&:=1+z\,q(z)\eins,
\intertext{where $\eins:=\mybrace{1,\ldots,1}^T$ and the row vector $q(z)$ is defined as}
q(z)&:=b^T\mybrace{\Id-z\Q}^{-1}.
\end{align*}
As is well known for a $p$th order Runge-Kutta scheme the stability function is an rational approximation of order $p$ to the exponential function \cite{Hairer2002}.

We assume the Runge-Kutta scheme to be $L$-stable which means that the stability function satisfies the criterion for $A$-stability, i.e.
\begin{align*}
\betrag{r(z)}\leq 1\quad\text{for all }\Realteil{z}\leq 0,
\end{align*}
as well as
\begin{align*}
\lim\limits_{\Realteil{z}\rightarrow -\infty}r(z)=0.
\end{align*}
Here the first property guarantees that the whole negative half-plane is contained in the region of stability. 

Additional to the $L$-stability, we make the following extra assumptions on the Runge-Kutta scheme which we list next.
\begin{assumption}\label{asRK} The $L$-stable Runge-Kutta scheme fulfils:
	\begin{itemize}	
	\item[(a)] The row vector of the weights equals the last row in the Runge-Kutta matrix.
	\item[(b)] The Runge-Kutta matrix $\Q$ is invertible.
	\item[(c)] The eigenvalues of the Runge-Kutta matrix $\Q$ have positive real part. 
\end{itemize}
\end{assumption}
Here assumptions (a) and (b) imply
\begin{equation*}
r(z)=b^T\Q^{-1}\mybrace{\Id-z\Q}^{-1}\eins=e_s^T\mybrace{\Id-z\Q}^{-1}\eins,
\end{equation*}
whereas (c) means that the spectrum $\sigma(\Q)$ is contained in $\sector{\theta_1}$.
The assumptions listed above are in particular satisfied by the Radau IIA familiy of Runge-Kutta methods; see \cite{Hairer2002}. It is well known that an $s$-stage Radau IIA method is of order $2s-1$ and stage order $s$ \cite{Hairer2002}. In the numerical experiments in Section \ref{chNumEx} we make use of these methods of order 1, 3 and 5. 
\subsection{Runge-Kutta based convolution quadrature}\label{chRKConvQuad}
As from now we restrict our considerations to the inhomogeneous evolution equation with zero initial data \eqref{eqInhomoDiffglg} whose strong solution is then solely given as a convolution. Because every equation in the form \eqref{eqDiffglg} can be easily transformed into an equation with zero initial data, this does not limit the applicability of our algorithm.

To construct a convolution quadrature we insert the contour integral representation \eqref{eqPa} of $\Pa$ into \eqref{eqLosu}. Interchanging the integrals yields
\begin{equation}\label{eqLosuylam}
u(t)=\frac{1}{2\pi\I}\int_\Gamma\res{\lambda^\alpha}{A} y_\lambda(t)\dint\lambda,
\end{equation} 
where $y_\lambda(t):=\int_0^t\e^{\lambda(t-\tau)}g(\tau)\dint\tau$ is the solution to the inhomogeneous differential equation $y'(t)=\lambda y(t)+g(t)$ with $y(0)=0$; see \cite{Pazy1983}. Here, following the idea of \cite{Lubich1993} we discretize this initial value problem by a Runge-Kutta method that satisfies Assumptions \ref{asRK}, insert the approximation into \eqref{eqLosuylam} and simplify the result using Cauchy's integral formula. Thus, we obtain the Runge-Kutta based convolution quadrature via a generating function as it is formulated in \cite{Lubich1993}:

\begin{proposition}\label{satzFQDlubich}
	Consider an $s$-stage Runge-Kutta method that satisfies Assumptions \ref{asRK}. Then at time $t_N=Nh$, where $N$ is the number of steps and $h$ is the step size with $Nh\leq T$, we get the approximation to the solution of \eqref{eqInhomoDiffglg} by
	\begin{equation*}
	u_{N}=h\,\mybrace{e_s^T\otimes \Id}\sum\limits_{n=0}^{N-1}W_{n}\,G_{N-1-n},
	\end{equation*}
	where $G_n=(g(t_n+c_kh))_{k=1}^s$ is a column vector and $W_n$ is the convolution quadrature weight given as the $n$th coefficient of the generating function
	\begin{align}\label{eqWn}
	h\sum\limits_{n=0}^\infty W_n\zeta^n=\frac{h}{2\pi\I}\int_\Gamma\mybrace{\Delta(\zeta)-h\lambda\Id}^{-1}\otimes\res{\lambda^\alpha}{A}\dint\lambda=\res{\mybrace{\frac{\Delta(\zeta)}{h}}^\alpha}{A}
\end{align} with
\begin{align*}
	\Delta(\zeta):=\mybrace{\Q+\frac{\zeta}{1-\zeta}\eins b^T}^{-1}.
	\end{align*}
\end{proposition}
As it is shown in \cite{Lubich1993} the convolution quadrature weights satisfy for $h\leq h_0$ with a sufficiently small $h_0$ and constants $C>0,\ \gamma\geq 0$ the estimate
\begin{equation*}
\norm{W_n}\leq C\,(nh)^{\alpha-1}\, e^{\gamma\, nh}\quad\text{for}\quad n\geq 1,
\end{equation*}
where for $n=0$ the same bound holds as for $n=1$, which implies the stability of the approximation. Furthermore, under some assumptions on the function $g$ and its derivatives the convolution quadrature is convergent of order $\min(p,\ \widetilde{p}+1+\alpha)$:
\begin{theorem}\label{thmKonv}
	Assume \eqref{defSectorial} and consider a Runge-Kutta method of order $p$ and stage order $\widetilde{p}$ that satisfies Assumptions $\ref{asRK}$. Then the error of the convolution quadrature at $t_N=Nh$ is bounded for $h\leq h_0$ with a sufficiently small $h_0$ by 
	\begin{align*}
	\norm{u_N-u(t_N)}&\leq C\,h^p\,\sum\limits_{l=0}^{\widetilde{p}}\mybrace{1+t_N^{\alpha+l-p}}\norm{g^{(l)}(0)}\\
	&+C\mybrace{h^p+h^{\widetilde{p}+1+\alpha}\betrag{\log(h)}}\mybrace{\sum\limits_{l=\widetilde{p}+1}^{p-1}\norm{g^{(l)}(0)}+\max\limits_{0\leq\tau\leq t_n}\norm{g^{(p)}(\tau)} }.
	\end{align*}
	The constants $C$ and $h_0$ depend only on the Runge-Kutta method, on the constants in \eqref{defSectorial} and on the length of the time interval. Especially they are neither affected by $N$ and $h$ with $Nh\leq T$ nor by $g\in\Cfkt{p}{\Rplus,X}$.
\end{theorem} 

The proof of Theorem \ref{thmKonv} can be found in \cite{Lubich1993}. It is also shown there that the equation
\begin{equation*}
\mybrace{\Delta(\zeta)-z\Id}^{-1}=\Q\mybrace{\Id-z\Q}^{-1}+\sum\limits_{n=1}^\infty r(z)^{n-1}\mybrace{\Id-z\Q}^{-1}\eins b^T\mybrace{\Id-z\Q}^{-1}\zeta^n
\end{equation*}
holds under Assumptions \ref{asRK}. On inserting this into Cauchy's integral formula \eqref{eqWn} we get a representation of the last ``row" of the convolution quadrature weights, i.e.
\begin{align}\label{eqwj}
w_n:=\mybrace{e_s^T\otimes\Id}W_n=\frac{1}{2\pi\I}\int_\Gamma r(h\lambda)^n\, q(h\lambda)\otimes\res{\lambda^\alpha}{A}\dint\lambda\quad\text{for}\quad n\geq 0.
\end{align}
The norm of the weights $w_n$ is bounded by the same term as the one of the weights $W_n$ except for a different constant $C$. Together with the fact that $r(h\lambda)^n$ is an approximation to $\exp(nh\lambda)$, this shows that \eqref{eqwj} can be interpreted as a discrete analogue of the operator $\Pa$ \eqref{eqPa}. 

Now we are able to reformulate Proposition \ref{satzFQDlubich} using the contour integral representation instead of generating functions.
\begin{proposition}\label{satzFQD}
For an $s$-stage Runge-Kutta method that satisfies Assumptions \ref{asRK} we get at time $t_N=Nh$ the approximation	
\begin{align*}
u_{N}=h\sum\limits_{n=0}^{N-1} w_n\,G_{N-1-n}
=h\sum\limits_{n=0}^{N-1}\mybrace{\frac{1}{2\pi\I}\int_\Gamma r(h\lambda)^n\, q(h\lambda) \otimes \res{\lambda^\alpha}{A} \dint\lambda}G_{N-1-n}, 
\end{align*}
where $\Gamma\in\sector{\theta_1}$ is chosen such that $h\Gamma\cap\sigma\!\mybrace{\Q^{-1}}=\emptyset$.
\end{proposition}

In the fast algorithm both Propositions \ref{satzFQDlubich} and \ref{satzFQD} will be applied.
\begin{remark}
Instead of basing convolution quadrature on Runge-Kutta methods it is also possible to construct it using multistep methods, e.g. backward differentiation formulas. For further details see \cite{Lubich1986, Schaedle2006}. 
\end{remark}
\section{Discretisation of the contour integrals}\label{chDiscret}
Relevant to the fast algorithm is the question how to discretise the contour integrals \eqref{eqwj} along suitable complex contours. In this section this issue will be discussed following the approach in \cite{Lopez-Fernandez2005}, \cite{Lopez-Fernandez2006} and \cite{Schaedle2006}.
\subsection{Quadrature using hyperbolas}\label{chQuadHyp}
Since it is neither useful to discretise the integrals $w_n$ for every $n$ with the same contour, nor efficient to define a different contour for each integral, we consider the sequence of fast-growing intervals
\begin{equation*}
	\Il=\lgro{\Lambda^{\ell-1}h,\,\Lambda^\ell h},\quad\ell\geq 1,
\end{equation*}
where $\Lambda>1$ is an integer, and fix one contour $\Gl$ for each such interval. If $nh\in\Il$, then \eqref{eqwj} is approximated by a quadrature along $\Gl$. Here, e.g. $\Lambda=5$ showed up as a good choice in the numerical experiments.

Recalling that the operator $A$ is sectorial we choose the contours $\Gl$ as hyperbolas which are parameterised by the maps
\begin{equation}\label{eqHyperbel}
	\gamma_\ell:\ \R\rightarrow \Gl,\quad x\mapsto\gamma_\ell(x)=\mu_\ell\,(1+\sin(ix-\varphi)),\quad\ell\geq 1,
\end{equation}
where $\varphi>0$ is the angle between the upper imaginary axis and the asymptote on its left. Assuming $\varphi\in\loro{0,\theta}$ with $\theta<\theta_1$ the hyperbolas $\Gl$ hence are all contained in the sector of analyticity $\sector{\theta_1}$. In contrast to the angle, the scale parameter $\mu_\ell>0$ depends on $\ell$. Furthermore, the hyperbolas are chosen such that the singularities of $r(h\lambda)^n\,q(h\lambda)$ lie to the right of the contour.

Since we consider an $L$-stable Runge-Kutta scheme, the integrand in \eqref{eqwj} decays rapidly for $\Realteil{\lambda}\rightarrow-\infty$ what makes the integral well-suited for approximations by the trapezoidal rule; see \cite{Weideman2007}. Applying this rule to \eqref{eqwj} with contour $\Gl$ and choosing the number of quadrature points on $\Gl$, independent of $\ell$, as $2K+1$ yields
\begin{align}\label{eqApproxwn}
	w_n\approx\sum\limits_{k=-K}^K\omega_k^{(\ell)}\,r(h\lamkl)^n\, q(h\lamkl)\otimes\res{(\lamkl)^\alpha}{A}\quad\text{for}\quad nh\in\Il,
\end{align}
where the weights $\omega_k^{(\ell)}$ and quadrature points $\lamkl$ are given by
\begin{equation*}
	\omega_k^{(\ell)}=\frac{\tau}{2\pi\I}\,\gamma_\ell\!'(x_k)\quad\text{and}\quad\lamkl=\gamma_\ell(x_k),
\end{equation*}
with $x_k=k\tau$ and the step length parameter $\tau$.

\subsection{Theoretical error bound}
To determine an appropriate number of quadrature points we take a look at the error $E(\tau,K,h,n)$ of the contour integral approximation in \eqref{eqApproxwn}. For the case of the hyperbola the following theoretical error bound, that shows exponential convergence, can be found in \cite{Schaedle2006}.
\begin{theorem}\label{thmFehlerschatzung}
	There are positive constants $C,d,c_0,\ldots,c_4$ and $b$, so that, if \linebreak${1\leq b\mu t\leq n}$, the quadrature error in \eqref{eqApproxwn} for a hyperbola \eqref{eqHyperbel} at ${t=nh\leq T}$ is bounded by
	\begin{align}
	\norm{E(\tau,K,h,n)}\leq C\,t^{\alpha-1}\,(\mu t)^{1-\alpha}\left(\frac{\e^{c_0\mu t}}{\e^{\nicefrac{2\pi d}{\tau}}-1}\right. &+ \e^{(c_1-c_2\cosh(K\tau))\mu t}\label{eqQuadfehlerabsch}\\
	&+\, \e^{c_3\mu t}\left.\mybrace{1+\frac{c_4\cosh(K\tau)\mu t}{\nicefrac{n}{2}}}^{-\nicefrac{n}{2}}\right).\nonumber
	\end{align}
\end{theorem}

Here, the first summand in the estimate corresponds to the error in the discretisation of the integral \eqref{eqwj} which we obtain by an infinite quadrature series over all integers $k$ and where the contour $\Gl$ is parameterised over the real line with an integrand holomorphic in an horizontal strip $\curly{z\in\C\sd{\betrag{\Imagteil{z}}\leq d}}$; see \cite{Lopez-Fernandez2004}. Truncating the series to the terms $\kgkg{-K}{k}{K}$ leads to the other two terms in the sum \eqref{eqQuadfehlerabsch}.

By choosing the step size $\tau$ so small that ${c_0\mu t-(\nicefrac{2\pi d}{\tau})\leq \log(\varepsilon)}$, the first term in \eqref{eqQuadfehlerabsch} becomes $\Oh{\varepsilon\,t^{\alpha-1}}$. In order to do so we require an asymptotic proportionality ${\frac{1}{\tau}\sim\log\mybrace{\nicefrac{1}{\varepsilon}}+\mu t}$. If we additionally choose $\mu$ such that
\begin{equation}\label{eqAbschmut}
	\frac{a_1}{\Lambda}\log\mybrace{\frac{1}{\varepsilon}}\leq \mu t\leq a_1\log\mybrace{\frac{1}{\varepsilon}}
\end{equation}
for an arbitrary $a_1>0$ and with $\Lambda>1$, and if $c_1-c_2\cosh(K\tau)\leq \nicefrac{-\Lambda}{a_1}$, then the next summand as well is $\Oh{\varepsilon\,t^{\alpha-1}}$. Here, the latter condition holds with $\cosh(K\tau)=a_2$ for a sufficiently large constant $a_2$. With the above choice of $\tau$ this yields $K\sim\log\mybrace{\nicefrac{1}{\varepsilon}}$. For $n\geq a_3\log(\nicefrac{1}{\varepsilon})$ and a sufficiently large constant $a_3$, the third term then becomes smaller than $\varepsilon\,t^{\alpha-1}$.

All in all, we get the following bound for the required number of quadrature points on the hyperbola; cf. \cite{Schaedle2006}.
\begin{theorem}\label{thmKlogeps}
If \eqref{eqAbschmut} is satisfied for $a_1>0$ and $\Lambda>1$, then in \eqref{eqApproxwn} a quadrature error bounded in norm by $\varepsilon\,t^{\alpha-1}$ is obtained for  
	\begin{equation*}
	K=\Oh{\log\mybrace{\frac{1}{\varepsilon}}}.
	\end{equation*} 
	The estimate holds for $n\geq a_3\log(\nicefrac{1}{\varepsilon})$ with a sufficiently large constant $a_3>0$. Here the number $K$ is independent of $\ell,\ n$ and $h$ with $nh\leq T$.
\end{theorem}

Indeed we do not obtain a small error bound for the first few $n<a_3\log(\nicefrac{1}{\varepsilon})$ as shown in numerical experiments in \cite{Schaedle2006}. Therefore, these first quadrature weights are computed in a different way that is described in Section \ref{chFirstWeights}.

\subsection{Choice of the parameters}\label{chParameter}
Concerning the inverse Laplace transform and its discretisation along hyperbolas, strategies for choosing parameters are derived in \cite{Lopez-Fernandez2006} and \cite{Weideman2007} under the assumption $t\in\lgro{t_0,\Lambda t_0}$. An error estimate is obtained there that is explicit in all constants involved. Achieving such an error estimate for the discretisation \eqref{eqApproxwn} is much more complicated as the convolution quadrature weights $w_n$ are not given as inverse Laplace transforms, but can just be interpreted as a discrete analogue of those, cf. Section \ref{chRKConvQuad}. Since for large $n$ and small $h$ with $nh\leq T$ the estimate in Theorem \ref{thmFehlerschatzung} tends to an expression of the same type as the error estimate for the approximation of the inverse Laplace transform in \cite{Lopez-Fernandez2006}, we can nevertheless in practice choose the parameters according to \cite{Lopez-Fernandez2006}.

Therefore, for the approximation interval $\Il$ and $2K+1$ nodes on the hyperbola as in Theorem \ref{thmKlogeps}, we use the following strategy to determine parameters for the discretisation:
\begin{enumerate}
	\item Choose $\varphi=d=\nicefrac{\theta}{2}$ with $\theta<\theta_1$.
	\item Minimize for $0<\rho<1$ the expression
	\begin{equation*}
	\eps\,\epsilon_K(\rho)^{\rho-1}+\epsilon_K(\rho)^\rho,
	\end{equation*}
	where
	\begin{equation*}
	\epsilon_K(\rho)=\exp\mybrace{\frac{-2\pi d}{a(\rho)}K},\quad a(\rho)=\arccosh\mybrace{\frac{\Lambda}{(1-\rho)\sin(\varphi)}}
	\end{equation*}
	and $\eps$ is the machine precision.
	\item Take 
	\begin{equation*}
	\tau=\frac{1}{K}a(\rho_\mathrm{opt})\quad\text{and}\quad \mu_\ell=\frac{2\pi d\, K\, (1-\rho_\mathrm{opt})}{\Lambda^\ell h\, a(\rho_\mathrm{opt})}.
	\end{equation*}
\end{enumerate}
With this choice of parameters we are now able to draw the connection between Theorem \ref{thmKlogeps} and the intervals $\Il$:\\ If $t=nh\in\Il$, then $\Lambda^{\ell-1}\leq n <\Lambda^\ell$ and with $\mu_\ell$ as above and $a_1=\frac{2\pi d(1-\rho_\mathrm{opt})}{a(\rho_\mathrm{opt})}>0$ we obtain  
\begin{align}
\frac{a_1}{\Lambda}\,K  = \frac{a_1\,\Lambda^{\ell-1}}{\Lambda^\ell}\,K \leq \frac{a_1\,n}{\Lambda^\ell}\,K = \mu_\ell\, t < \frac{a_1\,\Lambda^\ell}{\Lambda^\ell}\,K=a_1\,K.\nonumber
\end{align}
Thus, \eqref{eqAbschmut} holds. With $K=\Oh{\log\mybrace{\nicefrac{1}{\varepsilon}}}$ Theorem \ref{thmKlogeps} implies that the quadrature error is bounded by $\varepsilon\,t^{\alpha-1}$ for $nh\in\Il$ independent of $\ell$, except for the first few $n$. Here, our choice of intervals $\Il$ comes in useful. 

\section{The fast and parallel algorithm}\label{chAlgo}
With the results of the previous section we are now able to describe the fast algorithm for computing the solution \eqref{eqLosu} after $N$ time steps with step size $h$.
\subsection{Computing the first summands of the approximation}\label{chFirstWeights}
We already know that the error bound in Theorem \ref{thmKlogeps} does not hold for all convolution weights and the approximation properties of the discretised contour integral are poor for the first few weights $w_n$ with $n\leq\kappa$ (e.g. $\kappa=20$ in our numerical experiments, or ${\kappa  = a_3\log(\nicefrac{1}{\varepsilon})}$ asymptotically). Therefore for $n\leq\kappa$ we consider the weights $W_n$ from Proposition \ref{satzFQDlubich} and make use of their representation as an integral over a circle with radius~$\rho$, i.e.
\begin{equation*}
hW_n=\frac{1}{2\pi\I}\int_{\betrag{\zeta}=\rho}\zeta^{-n-1}\res{\mybrace{\frac{\Delta(\zeta)}{h}}^\alpha}{A}\dint\zeta;\quad\text{see \cite{Lubich1993}},
\end{equation*} 
whose approximation by the trapezoidal rule yields
\begin{equation}\label{eqWneinzelndirekt}
hW_n\approx \frac{\rho^{-n}}{J}\,\sum\limits_{j=0}^{J-1}\res{\mybrace{\frac{\Delta(\zeta_j)}{h}}^\alpha}{A}\,\e^{\nicefrac{-2\pi\I nj}{J}}\quad\text{for}\quad  \vonbis{n}{0}{\kappa}
\end{equation} 
with $\zeta_j=\rho\,\e^{\nicefrac{2\pi\I j}{J}}$ and the parameter $J$ which we choose according to \cite{Lubich1988,Lubich1993}: Assuming that the values of the Laplace transform are computed with an accuracy $\varepsilon$ and choosing $J=\kappa$ and $\rho^J=\sqrt{\varepsilon}$ an error of $\Oh{\sqrt{\varepsilon}}$ is obtained in \eqref{eqWneinzelndirekt}. If we even choose $J\geq \kappa\log(\nicefrac{1}{\varepsilon})$ and $\rho=\e^{-\gamma h}$ with $\gamma>0$, then the error becomes $\Oh{\varepsilon}$.

Different from \cite{Lubich1993} we do not use this formula to compute each $W_n$ in an extra step, but consider the sum
\begin{equation*}
h\,\sum\limits_{n=0}^{\kappa}W_n\,G_{N-1-n}\approx \sum\limits_{n=0}^{\kappa}\frac{\rho^{-n}}{J}\,\sum\limits_{j=0}^{J-1}\res{\mybrace{\frac{\Delta(\zeta_j)}{h}}^\alpha}{A}\,G_{N-1-n}\,\e^{\nicefrac{-2\pi\I nj}{J}},
\end{equation*}
where we use the eigenvalue decomposition of the $s\!\times\!s$-matrices $\nicefrac{\Delta(\zeta_j)}{h}$ into the product $\nicefrac{\Delta(\zeta_j)}{h}=U_j\, \mathcal{D}_j\,U_j^{-1}$ with diagonal matrix $\mathcal{D}_j$ to obtain
\begin{align}\label{eqWndirekt}
	h\sum\limits_{n=0}^{\kappa}W_n\,G_{N-1-n}\approx\sum\limits_{j=0}^{J-1}\mybrace{U_j\otimes\Id}x_j.
\end{align}
Here $x_j$ is the solution of the decoupled system
\begin{equation}\label{eqSystemdirFalt}
\mybrace{\mybrace{\mathcal{D}_j}^\alpha\otimes\Id-\Id_s\otimes A}x_j=\mybrace{U_j^{-1}\otimes\Id}\,\sum\limits_{n=0}^\kappa\frac{\rho^{-n}}{J}\,G_{N-1-n}\,\e^{\nicefrac{-2\pi\I nj}{J}}.
\end{equation}
In order to compute \eqref{eqWndirekt} we need to solve $sJ$ linear systems what can be done in parallel. Here, the $J$ right hand sides of \eqref{eqSystemdirFalt} are computed in $\Oh{J\log(J)}$ operations using fast Fourier transform.
\subsection{Fast convolution approximation}
In order to give the fast approximation of the quadrature weights $w_n$ for $n>\kappa$ we modify the intervals $\Il$ from Section \ref{chQuadHyp} so that they do not contain the first $\kappa+1$ quadrature weights. That means, we now consider for $\Lambda>1$ the sequence 
\begin{align*}
	\widetilde{I}_\ell=\lgro{\Lambda^{\ell-1}(\kappa+1)h,\,\Lambda^\ell(\kappa+1)h},\quad\ell\geq 1,
\end{align*}
where $L$ is the smallest integer such that $N\leq \Lambda^L(\kappa+1)$. Thus, we can split the approximation $u_N$ into the following $L+1$ sums
\begin{equation*}
	u_N=h\sum\limits_{n=0}^{N-1} w_n\,G_{N-1-n}=\un{0} + \un{1}+\ldots+\un{L},
\end{equation*}
where
\begin{equation}\label{equN0}
	\un{0}:= h\sum\limits_{n=0}^{\kappa}w_n\,G_{N-1-n}=\mybrace{e_s\otimes \Id}\,h\sum\limits_{n=0}^\kappa W_n\,G_{N-1-n} 
\end{equation}
is approximated via \eqref{eqWndirekt}. For $\vonbis{\ell}{1}{L}$ we define the sums $\un{\ell}$ corresponding to the intervals $\widetilde{I}_\ell$ and hyperbolas $\Gl$ given in \eqref{eqHyperbel} as
\begin{align*}
	\un{\ell} := h\sum\limits_{nh\in \widetilde{I}_\ell}w_n\,G_{N-1-n}=h\sum\limits_{nh\in \widetilde{I}_\ell}\frac{1}{2\pi\I}\int_{\Gl}\mybrace{r(h\lambda)^n\,q(h\lambda)\otimes\res{\lambda^\alpha}{A}}\,G_{N-1-n}\dint\lambda,
\end{align*}
where we now choose $\mu_\ell=\frac{2\pi d\,K\,(1-\rho_\mathrm{opt})}{\Lambda^\ell\,(\kappa+1)\,h\,a(\rho_\mathrm{opt})}$ according to the results in Section \ref{chParameter}.

On discretising the convolution quadrature weights $w_n$ via \eqref{eqApproxwn} we obtain the approximation $\Un{\ell}$ to $\un{\ell}$ that, with $m_\ell:=\Lambda^{\ell}(\kappa+1)$ for $\vonbis{\ell}{0}{L-1}$ and $m_L:=N$, can be formulated as
\begin{align*}
	\Un{\ell}&=h\sum\limits_{n=m_{\ell-1}}^{m_\ell-1}\sum\limits_{k=-K}^K\omega_k^{(\ell)}\,\mybrace{r(h\lamkl)^n\,q(h\lamkl)\otimes\res{(\lamkl)^\alpha}{A}}\,G_{N-1-n}\\
	&=h \sum\limits_{k=-K}^K\omega_k^{(\ell)}\sum\limits_{n=N-m_{\ell}}^{N-1-m_{\ell-1}}\mybrace{r(h\lamkl)^{N-1-n}\,q(h\lamkl)\otimes\res{(\lamkl)^\alpha}{A}}\,G_{n}\\
	&=\sum\limits_{k=-K}^K\omega_k^{(\ell)}\,r(h\lamkl)^{m_{\ell-1}}\,\res{(\lamkl)^\alpha}{A}\,y_k^{(\ell)},
\end{align*} 
where
\begin{align*}
	y_k^{(\ell)}&= h \sum\limits_{n=N-m_{\ell}}^{N-1-m_{\ell-1}} \mybrace{r(h\lamkl)^{N-1-m_{\ell-1}-n}\,q(h\lamkl)\otimes\Id}\,G_{n}
\end{align*}
is the Runge-Kutta approximation to the solution at time $t=(N-m_{\ell-1})h$ of the linear initial-value problem
\begin{align}\label{eqRKschnell}
	y'(t)=\lamkl\,y(t)+g(t),\quad y((N-m_\ell)h)=0. 
\end{align}
Thus, in order to compute $\Un{1},\ldots,\Un{L}$ we require $N-\kappa-1$ Runge-Kutta steps for a total of $2K+1$ differential equations \eqref{eqRKschnell}. Bringing to mind that ${K=\Oh{\log(\nicefrac{1}{\varepsilon})}}$ that means the computation of $\Oh{N\log(\nicefrac{1}{\varepsilon})}$ Runge-Kutta steps.\\
With the solution $x_k^{(\ell)}$ of the linear system of equations
\begin{equation}\label{eqGlgsRK}
	\mybrace{(\lamkl)^\alpha\,\Id-A} \,x_k^{(\ell)}=y_k^{(\ell)}
\end{equation}
the approximation $\Un{\ell}$ is obtained as the linear combination
\begin{equation}\label{eqUN}
	\Un{\ell} = \sum\limits_{k=-K}^{K} c_k^{(\ell)}\,x_k^{(\ell)},\quad\text{where}\quad c_k^{(\ell)}=\omega_k^{(\ell)}\,r(h\lamkl)^{m_{\ell-1}}.
\end{equation}
Therefore, in general, it is required to solve $2K+1$ linear systems \eqref{eqGlgsRK} when computing $\Un{\ell}$. If the inhomogeneity \eqref{eqLosu} though is real-valued, the following consideration shows that this number reduces to $K+1$: Since the quadrature points $\lamkl$ lie symmetric with respect to the real axis, this kind of symmetry is inherited by $y_k^{(\ell)}$ and $x_k^{(\ell)}$. Together with $\omega_{-k}^{(\ell)}=\omega_{k}^{(\ell)}$ this leads to the cancellation of the imaginary parts in \eqref{eqUN} and therefore only the sum of the real parts of half the terms needs to be computed.

In addition to the previously stated number of Runge-Kutta steps, we thus only require $\Oh{\log(N)\log(\nicefrac{1}{\varepsilon})}$ solutions to the linear system \eqref{eqGlgsRK} to obtain an approximation with an accuracy up to $\varepsilon$ what can be done in parallel.

To put it in a nutshell, the fast algorithm consists of the following steps:
\begin{enumerate}
	\item Compute \eqref{equN0} using \eqref{eqWndirekt} and solving $sJ$ systems in parallel.
	\item Compute in parallel for $\vonbis{\ell}{1}{L}$ \eqref{eqUN} with \eqref{eqRKschnell} and \eqref{eqGlgsRK}. Here \eqref{eqRKschnell} can be computed in parallel for $k$ and for the entries of the inhomogeneity $g$ whereas \eqref{eqGlgsRK} can be parallelised over $k$.
	\item Obtain the approximation $U_N$ to $u_N$ as the sum
	\begin{equation*}
	U_N=\Un{0}+\Un{1}+\ldots+\Un{L}.
	\end{equation*} 
\end{enumerate}  
 
 \begin{remark}
 	If we consider an additional operator $M$ and assume that \linebreak${(\lambda M-A)^{-1}}$ is analytic in $\sector{\theta_0}$ and bounded there by $\norm{(\lambda M-A)^{-1}}\leq\nicefrac{C}{\betrag{\lambda}}$ for some $C>0$, then we can extend the algorithm to the fractional evolution equation
 	\begin{equation*}
 	M\CDiff{u(t)}{\alpha}=Au(t)+g(t)\quad\text{for}\quad t\in I\quad \text{with}\quad u(0)=u_0,
 	\end{equation*}
 	which can be found in the following numerical experiments. 
 	In \eqref{eqGlgsRK} now $x_k^{(\ell)}$ is the solution to the linear system $\mybrace{(\lamkl)^\alpha\,M-A} \,x_k^{(\ell)}=y_k^{(\ell)}$ and the left side of \eqref{eqSystemdirFalt} becomes $\mybrace{\mybrace{\mathcal{D}_j}^\alpha\otimes M-\Id\otimes A}x_j$.
 \end{remark}
 
\section{Numerical Experiments}\label{chNumEx}
We give the following examples in order to illustrate the application and behaviour of the fast and parallel algorithm and to show the paper's main results.
\subsection{Fractional evolution equation with $2\!\times\!2$-matrix $A$}\label{chExample1}
As a first example we consider the inhomogeneous evolution equation
\begin{align}
\CDiff{u(t)}{\nicefrac{1}{2}} = Au(t) + g(t) \quad\text{for}\quad t>0\quad \text{with}\quad A=\begin{bmatrix}-1&1\\-1&-1\end{bmatrix} \label{eqBsp1}
\end{align}
under the initial data $u(0)=0$. By using the function \texttt{fracdiff} in Maple 18, we compute the inhomogeneity $g(t)$ such that the exact solution of \eqref{eqBsp1} is given by
\begin{equation*}
u(t)=\begin{bmatrix}
\sin(2t)^6\\ \mybrace{\tfrac{1}{2}-\tfrac{1}{2}\cos\mybrace{\sqrt{5}t}}^6
\end{bmatrix}.
\end{equation*}
Thus, we obtain an inhomogeneity that is five times continuously differentiable in $I$ and each of these derivations vanishes at $t=0$. Just like the inhomogeneity in the next example, $g(t)$ includes Fresnel integrals; see \cite{Gautschi1972}, which we compute using the implementation provided in \cite{DErrico2012}.

To have a look on the convergence, we use $\Lambda=5$ in the numerical experiment and test two different numbers of quadrature points ($K=10$ and $K=25$). The parameters for the hyperbolas $\Gl$ are chosen as described in Section \ref{chParameter} where we set $\theta_1=\nicefrac{\pi}{2}$ according to \eqref{eqWahltheta1}. To avoid errors arising from the poor approximation of the first few quadrature weights, we compute $\un{0}$ with $\kappa=20$ as described in Section \ref{chFirstWeights}, where we set $J=160$. 

Figure \ref{figkonvergenz} shows the absolute errors $\norm{U_N-u(Nh)}_\infty$ at time $t=Nh=10$ for $K=10$ and $K=25$ and for the Radau IIA methods of orders 1, 3 and 5 versus the step size $h$. Here the dashed lines represent the theoretical orders of convergence for the Runge-Kutta based convolution quadrature which are 1, 3 and 4.5 respectively; cf. Theorem \ref{thmKonv}. For the choice of $K=10$ we notice an error saturation which is due to the insufficient approximation of the contour integrals. This saturation is prevented if we increase the number of quadrature points to e.g. $K=25$ as in the right figure.
\begin{figure}[htbp]
	\centering
	\includegraphics[width=0.95\textwidth]{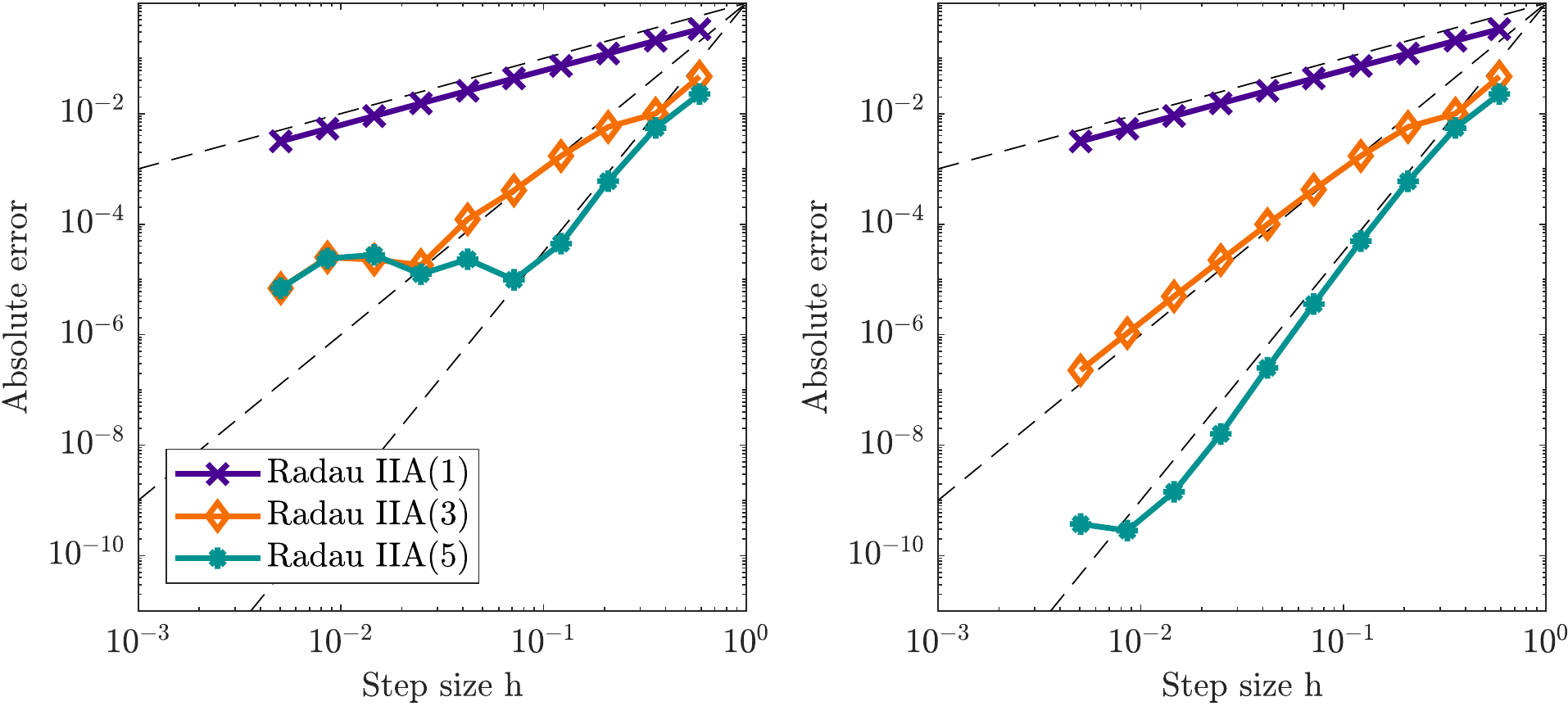}
	\caption{Absolute error versus time step size $h$ for three different Runge-Kutta methods, with $K=10$ (left) and $K=25$ (right).}\label{figkonvergenz}
\end{figure}

\subsection{Subdiffusion in three space dimension with periodic boundary conditions}\label{chExample2}
In order to illustrate the findings in Section \ref{chAlgo} we now study a subdiffusion equation in three space dimensions
\begin{align}
\CDiff{u(x,y,z,t)}{\nicefrac{1}{2}} =  \Delta_{x,y,z} u(x,y,z,t) + g(x,y,z,t)\ \ \text{for}\ \  x,y,z\in\loro{0,2\pi},\ t>0, \label{eqBsp2}
\end{align}
subject to the initial data $u(x,y,z,0)=0$ and with $2\pi$-periodic boundary conditions. We consider the inhomogeneity
\begin{align*}
g(x,y,z,t) = h_-(x,y,z)f_1(t)+h_+(x,y,z)f_2(t),
\end{align*}
where $f_1,f_2$ are functions that solely depend on time and \[h_\mp(x,y,z):=\cos(x)+\cos(y)+\cos(z)\mp (\sin(x)+\sin(y)+\sin(z)),\] such that the exact solution to \eqref{eqBsp2} is given by
\begin{align*}
u(x,t) = h_-(x,y,z)\,\sin(\pi t)-h_+(x,y,z)\,(\cos(\pi t)-1).
\end{align*}
For the spatial discretisation we make use of the following compact finite difference scheme that we, for reasons of clarity, state here for a one-dimensional problem:\\ Consider the differential equation $\Delta_{x}w(x)=f(x)$ with $x\in\loro{0,2\pi}$ under $2\pi$-periodic boundary conditions and an equidistant grid $x_j=j\eta$ in $\lgrg{0,2\pi}$ for a step size $\eta$. Then the scheme 
\begin{align}\label{eqkomFD}
\tfrac{1}{{\eta}^2} (w(x_{j-1})-2w(x_j)+w(x_{j+1}))\approx \tfrac{1}{12} f(x_{j-1})+\tfrac{5}{6}f(x_j)+\tfrac{1}{12}f(x_{j+1})
\end{align}
gives a fourth order approximation in space \cite{Lele1992}.

The execution time in seconds which is required to approximate the solution at time $t=123.45$ with varying numbers of time steps $N$ is shown in Figure \ref{figBarDiag}. Here, we distinguish between the time we need to compute the summand $\un{0}$ via \eqref{eqWndirekt}, to perform the necessary Runge-Kutta steps \eqref{eqRKschnell} using Radau IIA(5) and to solve the linear systems \eqref{eqGlgsRK}. Since the evaluation of the inhomogeneity in \eqref{eqRKschnell} is very expensive, we precompute the functions $f_1$ and $f_2$ for all required values of $t$.

 As in the previous example we set $\Lambda=5$, however, as we are interested in the time rather than accuracy, we chose $\kappa=12$, $J=\kappa+2$ and $K=20$. The computations are distributed to a pool of 8 workers using the Matlab Parallel Computing Toolbox.  We clearly see the linear growth of the time used for the Runge-Kutta steps and the logarithmic growth related to the linear systems whereas the computational work of the first summands of the approximation stays constant.
\begin{figure}[htbp]
	\centering
	\includegraphics[width=0.95\textwidth]{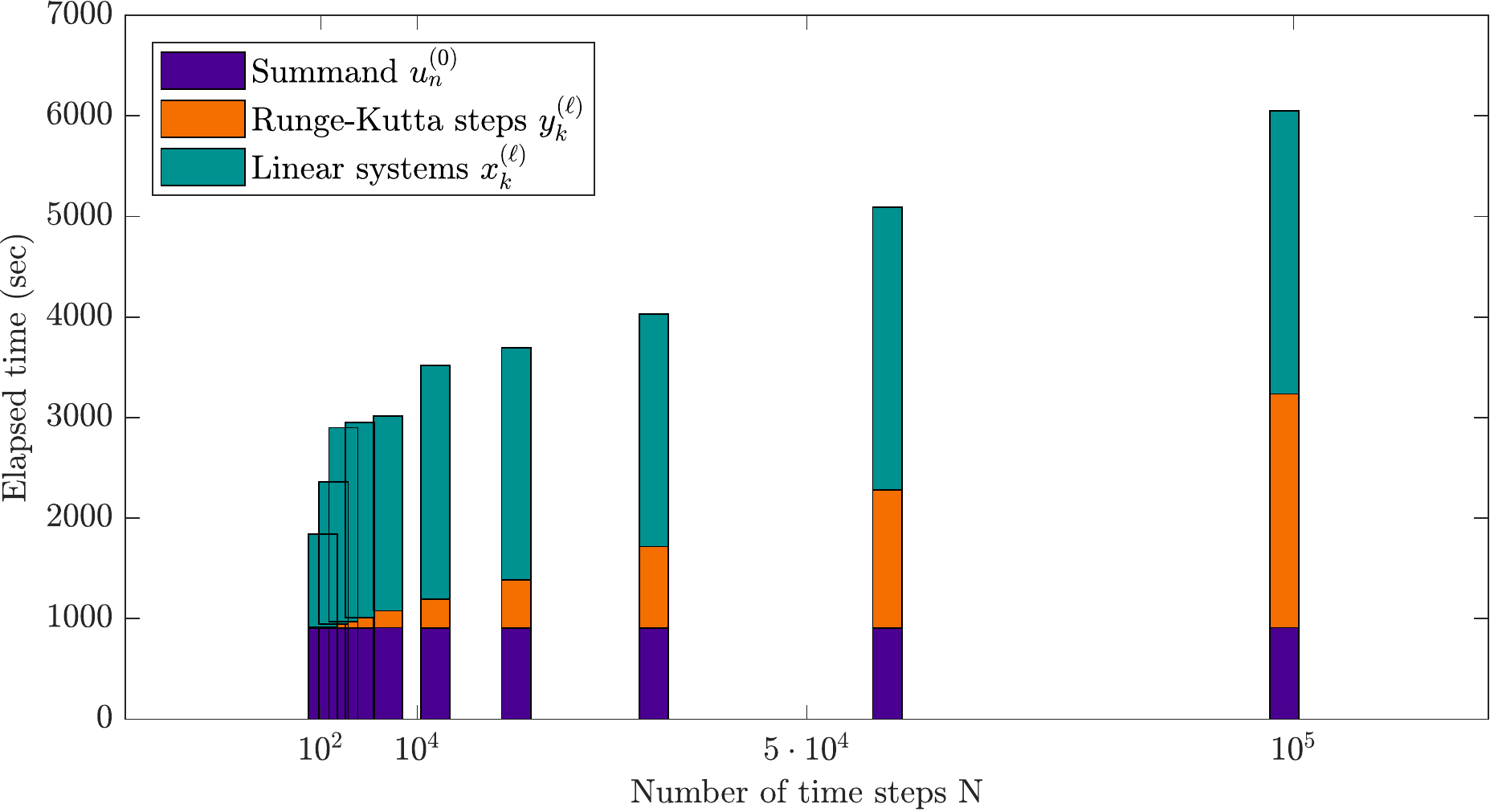}
	\caption{Elapsed time in seconds versus the number of time steps N.}\label{figBarDiag}
\end{figure}

Figure \ref{figPoolsize} illustrates the strong scaling of the algorithm. Here, we consider a fixed problem size but increase the number of parallel workers from 3 to a total of 32 workers. As in the previous figure we compute the solution at $t=123.45$, this time with a fixed number of steps $N=10^5$, and distinguish again between the time we need for the approximation of $\un{0}$, for the Runge-Kutta steps and for the solution to the linear systems. In the ideal case the time for the computation with $j$ workers would be $j$ times faster than the computation with one worker as it is shown in the dashed line. Whereas the Runge-Kutta steps and the solutions of the linear systems show a good parallelisation, the computation of the summand $\un{0}$ indicates an insufficient load balancing.
\begin{figure}[htbp]
	\centering
    \includegraphics[width=0.95\textwidth]{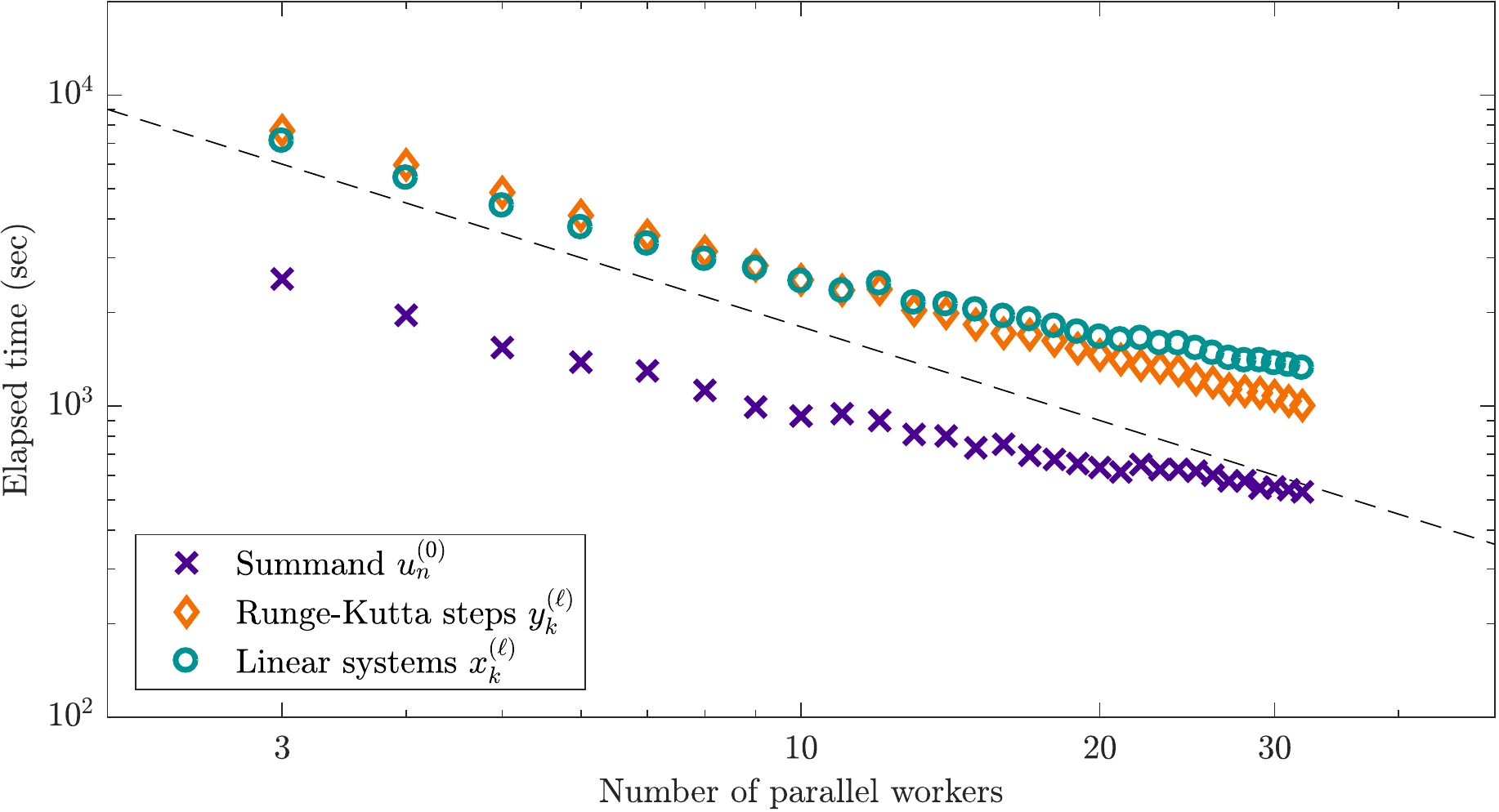}
    \caption{Elapsed time in seconds versus the number of workers which we distribute the computations to. The dashed line is the ideal line for strong scaling.}\label{figPoolsize}
\end{figure}

\subsection{Time-fractional Schr\"odinger equation with transparent boundary conditions}
The last example illustrates the algorithm's extensibility to more complicated problems. We consider the homogeneous fractional Schr\"odinger equation on the real line
\begin{align}\label{eqTBCv}
\CDiff{v(x,t)}{\alpha} = \I \Delta_{x} v(x,t)\quad\text{for}\quad x\in\R,\quad t>0
\end{align}
with an arbitrary $\alpha$ in $\loro{0,1}$ and the asymptotic condition $v(x,t)\rightarrow 0$ for $x\rightarrow\pm\infty$. The fractional equation is complemented with the initial condition $v(x,0)=u_0(x)$ for $x\in\R$. As it is known, the eigenvalues of the operator $\I\Delta_x$ lie on the negative imaginary axis \cite{Henry1981}, hence strictly speaking, it is not sectorial. However, the bound \eqref{eqAbschSaPa} is the one that is decisive for the application of the algorithm. Thus, to make sure that $\res{\lambda^\alpha}{\I\Delta_x}$ for $\lambda\in\sector{\theta_1}$ is analytic in $\sector{\nicefrac{\pi}{2}}$, we have to choose $\theta_1$ according to \eqref{eqWahltheta1}.

As mentioned in the beginning of Section \ref{chRKConvQuad}, we first transform equation \eqref{eqTBCv} into an equivalent differential equation with zero initial data by substituting $v = u+u_0$. Thus, we get the inhomogeneous fractional Schr\"odinger equation
\begin{align}\label{eqTBC}
\CDiff{u(x,t)}{\alpha} = \I\Delta_x u(x,t)+\I\Delta_x u_0(x)\quad\text{for}\quad x\in\R,\ t>0\quad\text{with}\quad u(x,0)=0.
\end{align}
In order to reduce the computation to a finite domain $\lgrg{-a,a}$ we assume $\mathrm{supp}(u_0(x))\subseteq\lgrg{-a,a}$, so that we can neglect the inhomogeneity $\I\Delta_x u_0(x)$ in the following derivation of the transparent boundary conditions. The method used for this purpose is well known; see e.g. \cite{Lubich2002} for the Schr\"odinger equation in the case of a discrete spatial Laplacian and \cite{Gao2012,Schaedle2006} for subdiffusion equations in the continuous case. It is based on the Laplace transform of \eqref{eqTBC} which reads
\begin{align}\label{eqBsp2Lapl}
\sa U(x,\lambda) = \I\Delta_x U(x,\lambda)\quad  \text{for}\quad \Realteil{\lambda}>0\quad \text{and}\quad \betrag{x}>a.
\end{align}
We only show the derivation of the right boundary condition at point $a$ in detail since the left condition at $-a$ is found analogously. That means, we consider equation \eqref{eqBsp2Lapl} with $x>a$, that spatially discretised with \eqref{eqkomFD} on a grid $x_j=j\eta$ with $j\geq N+1$ and $x_N=a$ for a fixed $N\in\N$ and step size $\eta$, becomes 
\begin{align*}
\varphi U^R(x_{j-1},s\lambda)+\psi U^R(x_j,\lambda)+\varphi U^R(x_{j+1},\lambda) &= 0\quad\text{for } j\geq N+1,
\end{align*}
with $\varphi=\frac{1}{12}\sa-\frac{\I}{\eta^2}$ and $\psi=\frac{5}{6}\sa+\frac{2\I}{\eta^2}$.  
The solution to the characteristic equation ${\varphi z^2+\psi z+\varphi=0}$ yields
\begin{align*}
z_{1,2}&=\frac{-\psi\pm\sqrt{\psi^2-4\varphi^2}}{2\varphi},
\end{align*}
where $\betrag{z_2}>1$ and $\betrag{z_1}<1$. Therefore, the decaying solution of the above three-term recursion is given by $U^R(x_j,\lambda)=U^I(a,\lambda)\,z_1^{N-j}$ for $j\geq N+1$. Here $U^I$ belongs to the spatial discretisation of the Laplace transform in time of \eqref{eqTBC} in the inner domain $\lgrg{-a,a}$.

To obtain a transparent boundary, we set the conditions
\begin{align*}
U^I(a,\lambda)=U^R(a,\lambda)\quad\text{and}\quad \delta_\nu U^I(a,\lambda)=\delta_\nu U^R(a,\lambda)
\intertext{with}\delta_\nu U(a,\lambda)=\delta_\nu U(x_N,\lambda)=\tfrac{1}{h}\mybrace{U(x_{N+1},\lambda)-U(x_N,\lambda)}.
\end{align*}
This leads to the equation
\begin{align*}
\frac{h}{z_1(\lambda)-1}\delta_\nu U^I(a,\lambda)=U^I(a,\lambda).
\end{align*}
Transforming back gives the transparent boundary condition at $x=a$, which can be formulated as the ``discrete Neumann-to-Dirichlet operator"
\begin{align}\label{eqTransBoundary}
u(a,t)=\int_{0}^{t}f(t-\tau)\,\delta_\nu u(a,\tau)\dint\tau,
\end{align}
where $f(t)$ is now the function with Laplace transform $F(\lambda)=\frac{h}{z_1(\lambda)-1}$.

In our numerical experiment we consider the fractional Schr\"odinger equation \eqref{eqTBC} with initial value $u_0(x)=10\exp({-(4x)^2+10\I x})$ and with a fractional derivative of order $\alpha=\nicefrac{3}{4}$. Hence, according to our previous considerations, we set $\theta_1=\nicefrac{\pi}{6}$ to make sure that \eqref{eqAbschSaPa} is satisfied for $\lambda\in\sector{\theta_1}$.

In order to check the correctness of the transparent boundary conditions \eqref{eqTransBoundary}, we compute the solution of equation \eqref{eqTBCv} on the finite domain $\lgrg{-2,2}$ with 801 spatial grid points and a fixed number of quadrature points $K=50$. As in the first example, we set $\Lambda=5$ and $\kappa=20$, where this time we choose ${J=4\kappa}$. To obtain a reference solution we consider the finite domain $\lgrg{-8,8}$, a spatial grid with 1601 grid points and set $\kappa=60$ and $J=4\kappa$ in the computation of the first convolution weights and $K=110$ for the nodes of the integration contour. In both cases, we make use of the convolution quadrature based on the Radau IIA(5) method and choose the parameters for the hyperbolas according to Section \ref{chParameter} with $\theta_1$ as discussed above. Figure \ref{figTBCSolErr} shows the modulus of the computed solution and the absolute error at times $t=0.05,0.1,\ldots,1$ with time step size $h=0.00025$ at each grid point in the domain $\lgrg{-2,2}$. 
\begin{figure}[htbp]
		\centering
		\includegraphics[width=0.95\linewidth]{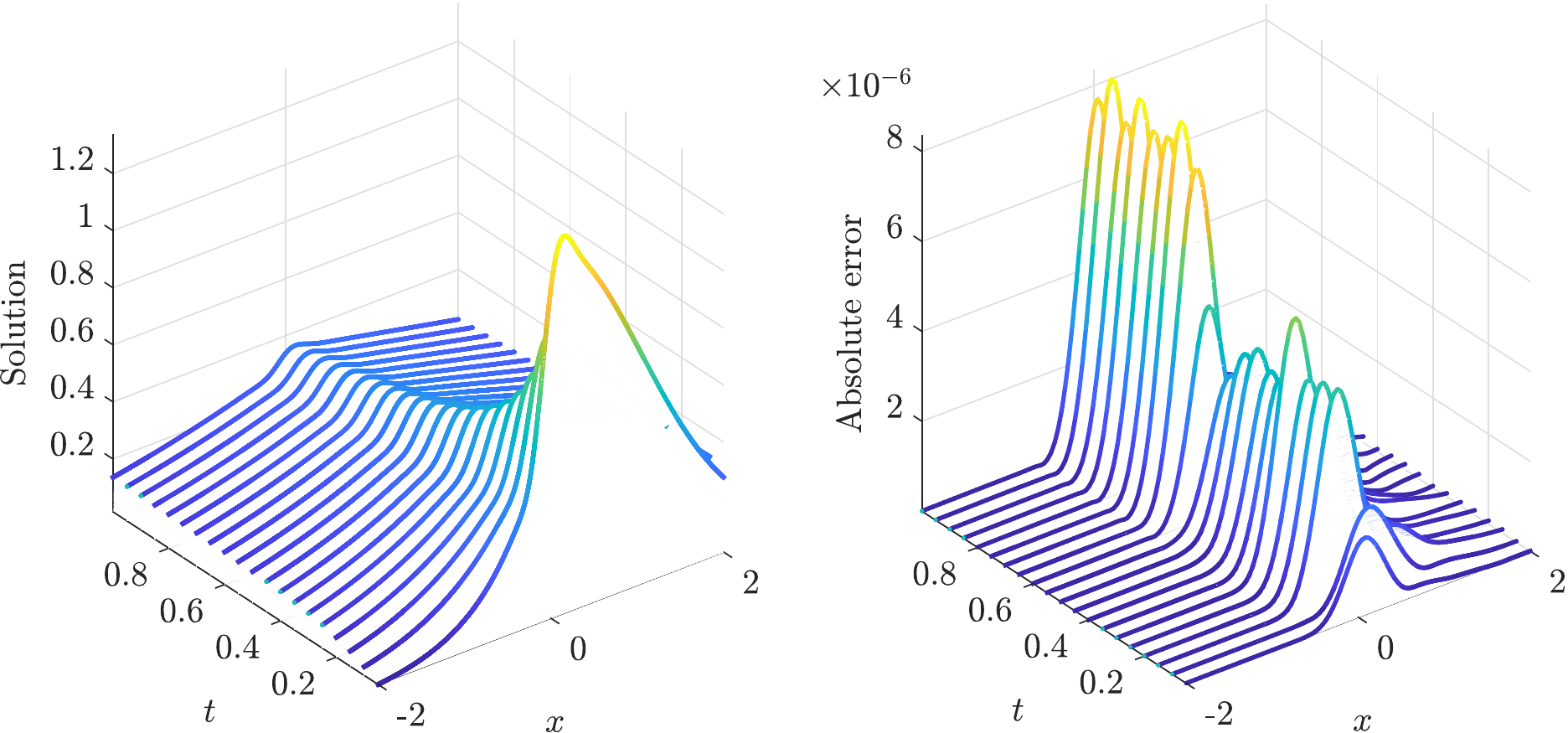}
	\caption{Modulus of the solution (left) and absolute error (right) at time $t=0.05,0.1,\ldots,1$ on the domain $\lgrg{-2,2}$ with $K=50$.}\label{figTBCSolErr}
\end{figure}

Whereas in Figure \ref{figTBCSolErr} we considered the fixed parameter $K=50$, Figure \ref{figTBCNCont1} plots the absolute error of the computed solution in dependence on the number of quadrature points. Here, the left side of the figure shows the solution at time $t=0.5$ with $6000$ time steps and different numbers of spatial grid points (201, 401, 801). For the right side we compute the solution at time $t=0.5$ with $401$ spatial grid points and take a look at different amounts of time steps $N$.
\begin{figure}[htbp]
	\centering
	\includegraphics[width=0.95\linewidth]{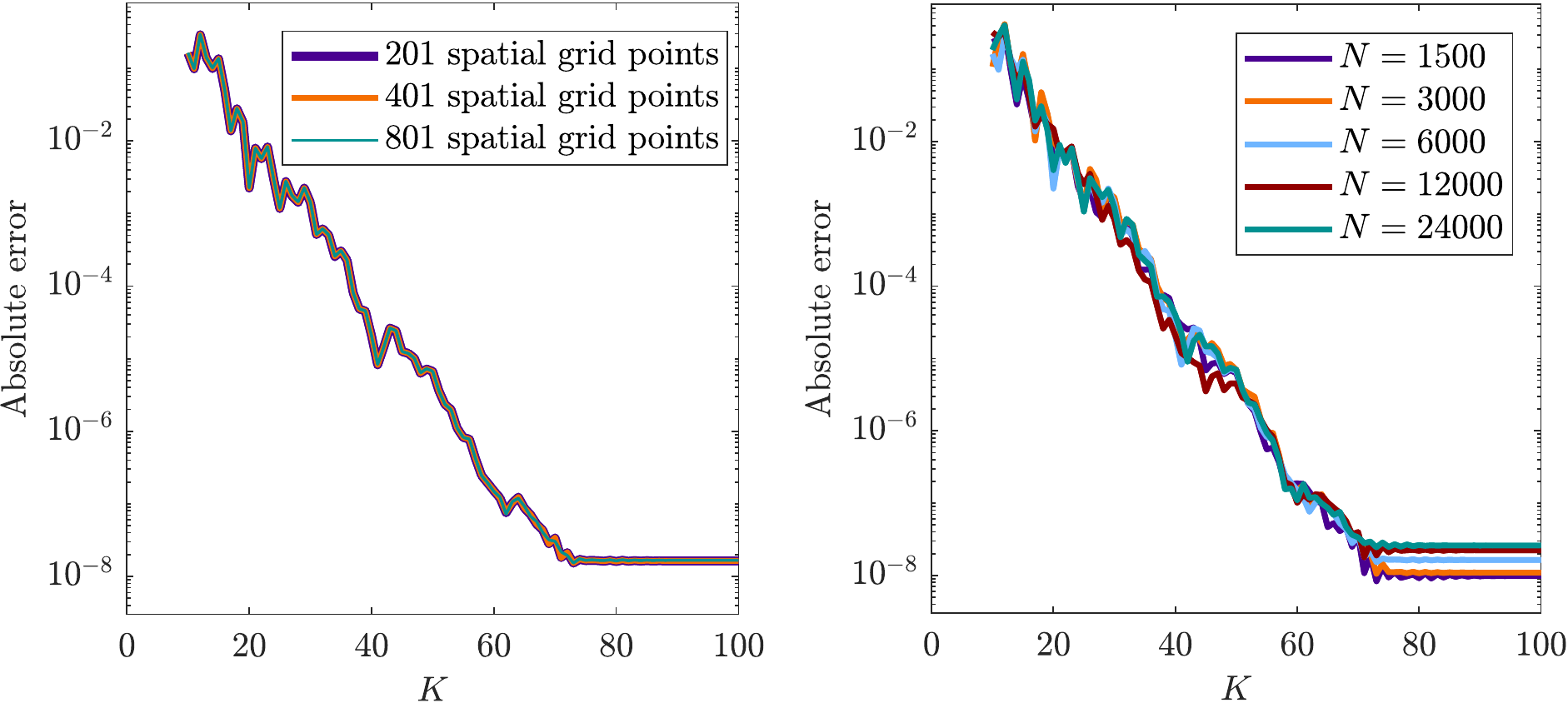}
	\caption{Absolute error at time $t=0.5$ versus $K$ for different pairs of time steps and numbers of spatial grid points. In the left figure we use $N=6000$ and in the right figure 401 grid points for the space discretisation.}\label{figTBCNCont1}
\end{figure}
For the reference solutions we use again $\Lambda=5,\ \kappa=110$ and $J=4\kappa$ and compute them such that the step size of the spatial discretisation as well as the time step size are identical to the ones of the computed solution. Thus, the error that we see in Figure \ref{figTBCNCont1} is dominated by the discretisation of the contour integral \eqref{eqApproxwn}. The exponential convergence that we expected due to Theorem \ref{thmFehlerschatzung} can be clearly seen. Here the error saturation that can be observed is related to the chosen value of $\kappa$.

\section*{Acknowledgement} I thank Achim Sch\"adle for the suggestion to write this paper, the helpful discussions during its preparation and for kindly providing his algorithm for the fast Runge-Kutta approximation of inhomogeneous parabolic equations.

\bibliography{lit}
\bibliographystyle{siamplain}
\end{document}